\documentclass[11pt]{article}
\usepackage{t1enc}
\usepackage[latin1]{inputenc}
\usepackage[english]{babel}
\usepackage{amsmath,amsthm}
\usepackage{amsfonts}
\usepackage{latexsym}
\usepackage[dvips]{graphicx}
\usepackage{graphicx}
\usepackage{color}

\usepackage{amsmath}
\usepackage{amssymb}
\usepackage{amsbsy}

\usepackage{amsfonts}

\usepackage{amsfonts,srcltx,mathrsfs}

\title{On extremal cacti with respect to the edge revised Szeged index}

\author{Shengjie  He$^{\rm a}$, Rong-Xia Hao$^{\rm a}$\footnote{Corresponding author.
Emails: he1046436120@126.com (Shengjie  He), rxhao@bjtu.edu.cn (Rong-Xia Hao), lidm@mail.cnu.edu.cn (Deming Li)}, Deming Li$^{\rm b}$\\
{\small\em $^{\rm a}$Department of Mathematics, Beijing Jiaotong University, Beijing,
100044, China}\\
{\small\em $^{\rm b}$Department of Mathematics, Capital Normal University, Beijing, 100048, China}\\
}

\date{} \textwidth 16cm \textheight 22cm \topmargin 0 cm \hoffset
-1.5 cm \voffset 0cm

\date{} \textwidth 16cm \textheight 22cm \topmargin 0 cm \hoffset
-1.5 cm \voffset 0cm

\newtheorem{theorem}{Theorem}[section]
\newtheorem{lemma}[theorem]{Lemma}

\begin{document}
\baselineskip 0.50cm \maketitle

\begin{abstract}
Let $G$ be a connected graph. The edge revised Szeged index of $G$ is defined as $Sz^{\ast}_{e}(G)=\sum\limits_{e=uv\in E(G)}(m_{u}(e|G)+\frac{m_{0}(e|G)}{2})(m_{v}(e|G)+\frac{m_{0}(e|G)}{2})$, where $m_{u}(e|G)$ (resp., $m_{v}(e|G)$) is the number of edges whose distance to vertex $u$ (resp., $v$) is smaller than the distance to vertex $v$ (resp., $u$), and $m_{0}(e|G)$ is the number of edges equidistant from both ends of $e$.
In this paper, we give the minimal and the second minimal edge revised Szeged index of cacti with order $n$ and $k$ cycles, and all the graphs that achieve the minimal and second minimal edge revised Szeged index are identified.

{\bf Keywords}: Szeged index;  Edge revised Szeged index; Cactus.

\end{abstract}

\section{Introduction}

Throughout this paper, all graphs we considered are finite, undirected, and simple. Let $G$ be a connected graph with vertex set $V(G)$ and edge set $E(G)$. For a vertex $u \in V(G)$, the degree of $u$, denote by $d_{G}(u)$, is the number of vertices
which are adjacent to $u$. Let $N_{G}(u)$ be the set of all neighbours of $u$ in $G$. Call a vertex $u$ a $pendant$ $vertex$ of $G$, if $d_{G}(u)=1$ and call an edge $uv$ a $pendant$ $edge$ of $G$, if $d_{G}(u)=1$ or $d_{G}(v)=1$.
An edge $e$ is called a $cut$ $edge$ of a connected graph $G$ if $G-e$ is disconnect. For any two vertices $u, v \in V(G)$, let $d_{G}(u, v)$ denote the distance between $u$ and $v$ in $G$. Denote by $P_n$, $S_n$ and $C_n$ a path, star and cycle on $n$ vertices, respectively.

The topological indices are quantity values closely related to chemical structure which can be used in theoretical chemistry for understanding the physicochemical properties of chemical compounds.
The Wiener index is one of the oldest and the most thoroughly studied topological index. The Wiener index of a graph $G$ is defined as
$$W(G)=\sum\limits_{\{ u, v\} \subseteq V(G) } d_{G}(u, v).$$

For any edge $e=uv$ of $G$, $V(G)$ can be partitioned into three sets by comparing with the distance of the vertex in $V(G)$ to $u$ and $v$, and the three sets are as follows:
$$N_{u}(e|G)=\{ w\in V(G): d_{G}(u, w) < d_{G}(v, w) \},$$
$$N_{v}(e|G)=\{ w\in V(G): d_{G}(v, w) < d_{G}(u, w) \},$$
$$N_{0}(e|G)=\{ w\in V(G): d_{G}(u, w) = d_{G}(v, w) \}.$$
The number of vertices of $N_{u}(e|G)$, $N_{v}(e|G)$, and $N_{0}(e|G)$ are denoted by $n_{u}(e|G)$, $n_{v}(e|G)$ and $n_{0}(e|G)$, respectively.
If $G$ is a tree, then the formula $W(G)=\sum\limits_{e=uv \in E(G)}n_{u}(e|G)n_{v}(e|G)$
gives known property of the Wiener index.

A new topological index, named by Szeged index, was introduced by Gutman \cite{Gut.A}, which is an extension of the Wiener index and defined by
$$Sz(G)=\sum\limits_{e=uv \in E(G)}n_{u}(e|G)n_{v}(e|G).$$

If $e=uv$ is an edge of $G$ and $w$ is a vertex of $G$, then the distance between $e$ and $w$ is defined as $d_{G}(e,w) = {\rm{min}} \{ d_{G}(u,w),d_{G}(v,w) \}$. For $e=uv \in E(G)$, let $M_u(e|G)$ be the set of edges whose distance to the vertex $u$ is smaller than the distance to the vertex $v$, $M_v(e|G)$ be the set of edges whose distance to the vertex $v$ is smaller than the distance to the vertex $u$, and $M_{0}(e|G)$ be the set of edges equidistant from both ends of $e$. Set $m_u(e|G)=|M_u(e|G)|$, $m_v(e|G)=|M_v(e|G)|$ and $m_0(e|G)=|M_0(e|G)|$.
The edge revised Szeged index \cite{HDZB} of a graph $G$ is defined as:
$$Sz^{\ast}_{e}(G)=\sum\limits_{e=uv\in E(G)}(m_{u}(e|G)+\frac{m_{0}(e|G)}{2})(m_{v}(e|G)+\frac{m_{0}(e|G)}{2}).$$
For other results on the Szeged index, we refer to \cite{Ao.M,C.LXL,Gut.A.R,LXue.M,zhang.H,ZhouB.X}.

A cactus is a connected graph that any block is either a cut edge or a cycle. It is also a graph in which any two cycles have at most one common vertex. A cycle in a cactus is called end-block if all but one vertex of this cycle have degree 2. If all the cycles in a cactus have exactly one common vertex, then they form a $bundle$. Let $\mathcal{C}(n,k)$ be the class of all cacti of order $n$ with $k$ cycles.
Let $\mathcal{C}_{0}(n,k) \in \mathcal{C}(n,k)$ be a bundle of $k$ triangles with $n-2k-1$ pendant edges attached at the common vertex of the $k$ triangles and
$\mathcal{C}_{1}(n,k) \in \mathcal{C}(n,k)$ be a bundle of $k$ quadrangles with $n-3k-1$ pendant edges attached at the
common vertex of the $k$ quadrangles (see Fig. 1).
For other research on cacti, we refer to \cite{LISC1, LISC2}.

\vspace{0.8cm}

\begin{center}   \setlength{\unitlength}{0.7mm}
\begin{picture}(30,65)

\put(-65,20){\circle*{1}}
\put(-55,20){\circle*{1}}
\put(-45,20){\circle*{1}}
\put(-20,20){\circle*{1}}
\put(-35,40){\circle*{1}}
\put(-55,60){\circle*{1}}
\put(-65,60){\circle*{1}}
\put(-50,60){\circle*{1}}
\put(-40,60){\circle*{1}}
\put(-15,60){\circle*{1}}
\put(-25,60){\circle*{1}}
\put(-65,60){\line(1,0){10}}
\put(-50,60){\line(1,0){10}}
\put(-25,60){\line(1,0){10}}
\put(-35,40){\line(-3,-2){30}}
\put(-35,40){\line(-1,-1){20}}
\put(-35,40){\line(-1,-2){10}}
\put(-35,40){\line(3,-4){15}}
\put(-35,40){\line(-3,2){30}}
\put(-35,40){\line(-1,1){20}}
\put(-35,40){\line(-3,4){15}}
\put(-35,40){\line(-1,4){5}}
\put(-35,40){\line(1,2){10}}
\put(-35,40){\line(1,1){20}}
\put(-39,18){$\cdots$}
\put(-32,18){$ \cdots $}
\put(-36,58.5){$\cdots$}
\put(-72,14){$n-2k-1$ pendant vertices}
\put(-53,67){$k$ triangles}
\put(-50,5){$\mathcal{C}_{0}(n,k)$}

\put(35,20){\circle*{1}}
\put(45,20){\circle*{1}}
\put(55,20){\circle*{1}}
\put(80,20){\circle*{1}}
\put(65,40){\circle*{1}}
\put(45,60){\circle*{1}}
\put(35,60){\circle*{1}}
\put(50,60){\circle*{1}}
\put(60,60){\circle*{1}}
\put(85,60){\circle*{1}}
\put(75,60){\circle*{1}}

\put(40,65){\circle*{1}}
\put(55,65){\circle*{1}}
\put(80,65){\circle*{1}}
\put(35,60){\line(1,1){5.1}}
\put(45,60){\line(-1,1){5.1}}

\put(50,60){\line(1,1){5.1}}
\put(60,60){\line(-1,1){5.1}}

\put(75,60){\line(1,1){5.1}}
\put(85,60){\line(-1,1){5.1}}

\put(65,40){\line(-3,-2){30}}
\put(65,40){\line(-1,-1){20}}
\put(65,40){\line(-1,-2){10}}
\put(65,40){\line(3,-4){15}}
\put(65,40){\line(-3,2){30}}
\put(65,40){\line(-1,1){20}}
\put(65,40){\line(-3,4){15}}
\put(65,40){\line(-1,4){5}}
\put(65,40){\line(1,2){10}}
\put(65,40){\line(1,1){20}}
\put(61,18){$\cdots$}
\put(68,18){$ \cdots $}
\put(64,58.5){$\cdots$}
\put(18,14){$n-3k-1$ pendant vertices}
\put(47,70){$k$ quadrangles}
\put(50,5){$\mathcal{C}_{1}(n,k)$}

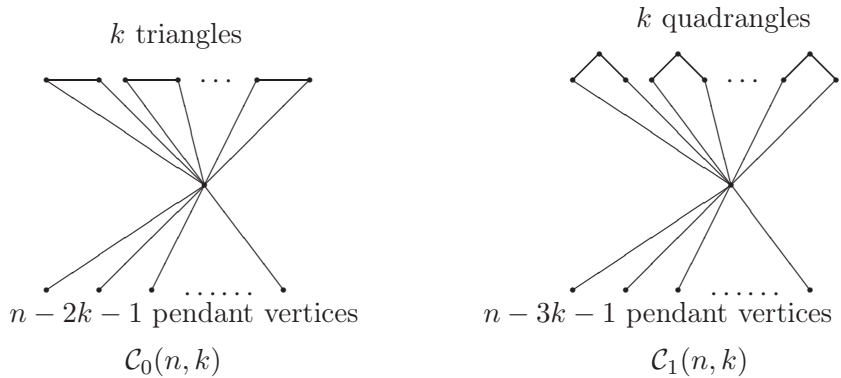
\put(-25,-5){Fig. 1. $\mathcal{C}_{0}(n, k)$ and $\mathcal{C}_{1}(n, k)$}

\end{picture} \end{center}

\vspace{0.1cm}

The edge revised Szeged index was introduced by Dong et al. \cite{HDZB} as a new deformation of Szeged index. Liu and Chen \cite{LMM} given an upper bound of the edge revised Szeged index for a connected bicyclic graphs.
In \cite{MHKH}, a matrix method was used to obtain the exact
formulae for computing the Szeged index of join and composition of graphs.
In \cite{Fag.M}, the edge revised Szeged index of the Cartesian product of graphs was computed.
Liu et al. in \cite{Yu} characterized the graph with minimum Szeged index among all the unicyclic graphs with given order and diameter.
Wang \cite{WANGSJ} determined the lower bound on revised Szeged index of cacti with $n$ vertices and $k$ cycles.
In this paper, by using the methods similar to Wang \cite{WANGSJ}, the edge revised Szeged index of the cacti with $n$ vertices and $k$ cycles is studied.
Moreover, the lower bound on edge revised Szeged index of the cacti with given cycles is determined and the corresponding extremal graph is identified. Furthermore, the second minimal edge revised Szeged index of the cacti with given cycles is established as well.

\section{Useful lemmas}

In this section, we will introduce three useful lemmas which will be used frequently in next sections.

Let $G$ be a graph with $|E(G)|=m$. For any edge $e=uv$, by the fact that $m_{u}(e|G)+m_{v}(e|G)+m_{0}(e|G)=m$, we have that
\begin{eqnarray*}
Sz^{\ast}_{e}(G)&=&\sum\limits_{e=uv\in E(G)}(m_{u}(e|G)+\frac{m_{0}(e|G)}{2})(m_{v}(e|G)+\frac{m_{0}(e|G)}{2})\\
&=&\sum\limits_{e=uv\in E(G)}(\frac{m+m_{u}(e|G)-m_{v}(e|G)}{2})(\frac{m+m_{v}(e|G)-m_{u}(e|G)}{2})\\
&=&\sum\limits_{e=uv\in E(G)}\frac{m^{2}-(m_{u}(e|G)-m_{v}(e|G))^{2}}{4}\\
&=&\frac{m^{3}}{4}-\frac{1}{4}\sum\limits_{e=uv\in E(G)}(m_{u}(e|G)-m_{v}(e|G))^{2}.\\
\end{eqnarray*}

For any edge $e=xy\in E(G)$, by using the fact that $m_{x}(e|G)+m_{y}(e|G)+m_{0}(e|G)=m$ and $e\in M_{0}(e|G)$, one has that $m_{x}(e|G)+m_{y}(e|G) \leq m-1$, $|m_{x}(e|G)-m_{y}(e|G)| \leq m-1$, and the last equality holds if and only if $e=xy$ is a pendant edge. Thus, we have the following Lemma~\ref{lem2.1}.

\begin{lemma}\label{lem2.1}
Let $e=uv$ be an edge of $G$ and $|E(G)|=m$. Then
$$(m_{u}(e|G)-m_{v}(e|G))^{2}\leq (m-1)^{2}$$
with equality if and only if $e=uv$ is a pendant edge.
\end{lemma}

\begin{lemma}\label{lem2.2}
Let $G$ be a graph with an even cycle $C_{2k}$ such that $G-E(C_{2k})$ has exactly $2k$ connected components and $|E(G)|=m$. Then
$$\sum\limits_{e=uv\in C_{2k}}(m_{u}(e|G)-m_{v}(e|G))^{2} \leq 2k(m-2k)^{2} $$
with equality if and only if $C_{2k}$ is an end-block.
\end{lemma}

\begin{proof}
Let $V(C_{2k})= \{ u_{1}, u_{2},  \cdots,  u_{2k}\}$ and $C_{2k}=u_{1}u_{2} \cdots u_{2k}u_{1}$. For each $1 \leq i \leq 2k$, let $G_{i}$ be the component of $G-E(C_{2k})$ that contains $u_{i}$.
Denote $m_{i}=|E(G_{i})|$, then we have that $\sum_{j=1}^{2k}m_{j}=m-2k$. For any edge $u_{i}u_{i+1}$, it can be checked that
$M_{_{u_{i}}}(u_{i}u_{i+1}|G)=E(G_{i}) \cup E(G_{i-1}) \cup \cdots E(G_{i-k+1}) \cup \{u_{i}u_{i-1}, u_{i-1}u_{i-2}, \cdots, u_{i-k+2}u_{i-k+1} \}$
and $M_{_{u_{i+1}}}(u_{i}u_{i+1}|G)=E(G_{i+1}) \cup E(G_{i+2}) \cup \cdots E(G_{i+k}) \cup \{u_{i+1}u_{i+2}, u_{i+2}u_{i+3}, \cdots, u_{i+k-1}u_{i+k} \}$
(where the subscripts are taken modulo $2k$).
Hence, it can be checked that
\begin{align}
&[m_{u_{i}}(u_{i}u_{i+1}|G)-m_{u_{i+1}}(u_{i}u_{i+1}|G)]^{2}\nonumber\\
&=(m_{i}+m_{i-1}+\cdots+m_{i-k+1}+k-1-m_{i+1}-m_{i+2}-\cdots-m_{i+k}-k+1)^{2}\nonumber\\
&\leq (\sum\limits_{j=1}^{2k}m_{j})^{2}\\
&=(m-2k)^{2}\nonumber.
\end{align}
Note that equality (1) holds if and only if $m_{j}=0$ for all $j=i, i-1, \cdots, i-k+1$ or $m_{j}=0$ for all $j=i+1, i+2, \cdots, i+k$.

On the other hand, if there are at least two positive integers among $m_1,m_2,\cdots,m_{2k}$, say $m_{a} > 0$ and $m_{b} > 0$, without loss of generality, assume $a<b$. Let $c= \lfloor  \frac{a+b}{2} \rfloor$. It can be checked that $E(G_{a}) \in M_{_{u_{c}}}(u_{c}u_{c+1}|G)$ and $E(G_{b})\in M_{_{u_{c+1}}}(u_{c}u_{c+1}|G)$.
Thus, one has that
$$(m_{u_{c}}(u_{c}u_{c+1}|G)-m_{u_{c+1}}(u_{c}u_{c+1}|G))^{2} < (m-2k)^{2}.$$
Thus, $\sum\limits_{e=uv\in C_{2k}}(m_{u}(e|G)-m_{v}(e|G))^{2} = 2k(m-2k)^{2} $ if and only if there is at most one positive integer among $m_1,m_2,\cdots,m_{2k}$, i.e., $C_{2k}$ is an end-block.
\end{proof}

\begin{lemma}\label{lem2.3}
Let $G$ be a graph with an odd cycle $C_{2k+1}$ such that $G-E(C_{2k+1})$ has exactly $2k+1$ connected components and $|E(G)|=m$. Then
$$\sum\limits_{e=uv\in C_{2k+1}}(m_{u}(e|G)-m_{v}(e|G))^{2} \leq 2k(m-2k-1)^{2} $$
with equality if and only if $C_{2k+1}$ is an end-block.
\end{lemma}

\begin{proof}
Let $C_{2k+1}=v_{1}v_{2} \cdots v_{2k+1}v_{1}$. For each $1 \leq i \leq 2k+1$, let $G_{i}$ be the component of $G-E(C_{2k+1})$ that contains $v_{i}$.
Denote $m_{i}=|E(G_{i})|$, then one has that $\sum_{j=1}^{2k+1}m_{j}=m-2k-1$. For any edge $v_{i}v_{i+1}$, it can be checked that
$M_{_{v_{i}}}(v_{i}v_{i+1}|G)=E(G_{i}) \cup E(G_{i-1}) \cup \cdots E(G_{i-k+1}) \cup \{v_{i}v_{i-1}, v_{i-1}v_{i-2}, \cdots, v_{i-k+1}v_{i-k} \}$
and $M_{_{v_{i+1}}}(v_{i}v_{i+1}|G)=E(G_{i+1}) \cup E(G_{i+2}) \cup \cdots E(G_{i+k}) \cup \{v_{i+1}v_{i+2}, v_{i+2}v_{i+3}, \cdots, v_{i+k}v_{i+k+1} \}$
(where the subscripts are taken modulo $2k+1$).
Hence, it can be checked that
\begin{align}
&(m_{v_{i}}(v_{i}v_{i+1}|G)-m_{v_{i+1}}(v_{i}v_{i+1}|G))^{2}\nonumber\\
&=(m_{i}+m_{i-1}+\cdots+m_{i-k+1}+k-m_{i+1}-m_{i+2}-\cdots-m_{i+k}-k)^{2}\nonumber\\
& \leq (\sum\limits_{j=1}^{2k+1}m_{j}-m_{i+k+1})^{2}\\
&=(m-2k-1-m_{i+k+1})^{2}\nonumber.
\end{align}
The equality (2) holds if and only if $m_{j}=0$ for all $j=i, i-1, \cdots, i-k+1$ or $m_{j}=0$ for all $j=i+1, i+2, \cdots, i+k$.

Thus, we have
\begin{align*}
&\sum\limits_{e=uv\in C_{2k+1}}(m_{u}(e|G)-m_{v}(e|G))^{2} \\
&\leq  \sum\limits_{i=1}^{2k+1}(m-2k-1-m_{i+k+1})^{2}\\
&=(2k+1)(m-2k-1)^{2}-2(m-2k-1)\sum\limits_{i=1}^{2k+1}m_{i+k+1}+\sum\limits_{i=1}^{2k+1}m_{i+k+1}^{2}\\
&=(2k+1)(m-2k-1)^{2}-2(m-2k-1)^{2}+\sum\limits_{i=1}^{2k+1}m_{i+k+1}^{2}\\
&\leq (2k+1)(m-2k-1)^{2}-2(m-2k-1)^{2}+(\sum\limits_{i=1}^{2k+1}m_{i+k+1})^{2}\\
&=(2k+1)(m-2k-1)^{2}-2(m-2k-1)^{2}+(m-2k-1)^{2}\\
&=2k(m-2k-1)^{2}.
\end{align*}

If there is at least two positive integers among $m_1,m_2,\cdots,m_{2k+1}$, say $m_{a} > 0$ and $m_{b} > 0$. Without loss of generality, assume $a<b$.
Let $c= \lfloor  \frac{a+b}{2} \rfloor$. It can be checked that $E(G_{a}) \in M_{_{u_{c}}}(u_{c}u_{c+1}|G)$ and $E(G_{b})\in M_{_{u_{c+1}}}(u_{c}u_{c+1}|G)$.
Then,
$$(m_{u_{c}}(u_{c}u_{c+1}|G)-m_{u_{c+1}}(u_{c}u_{c+1}|G))^{2} < (m-2k-1-m_{c+k+1})^{2}.$$

Thus, for each $1\leq i \leq 2k+1$, one has that
$$(m_{v_{i}}(v_{i}v_{i+1}|G)-m_{v_{i+1}}(v_{i}v_{i+1}|G))^{2}=(m-2k-1-m_{i+k+1})^{2}$$
if and only if there is at most one positive integer among $m_1,m_2,\cdots,m_{2k+1}$, i.e., $C_{2k+1}$ is an end-block. On the other hand,
if there is at most one positive integer among $m_1,m_2,\cdots,m_{2k+1}$, then
$$\sum\limits_{i=1}^{2k+1}m_{i+k+1}^{2}= (\sum\limits_{i=1}^{2k+1}m_{i+k+1})^{2}$$

The lemma holds immediately.
\end{proof}

\section{Cacti with minimum edge revised Szeged index in $\mathcal{C}(n, k)$}
Recall that a bundle is a cactus in which all cycles
have exactly one common vertex. Let $G(m_{1}, m_{2}, \cdots, m_{k})$ be a bundle of $k$ cycles with lengths $m_{1}, m_{2}, \cdots, m_{k}$, respectively, and with $n+k-1-\sum\limits_{i=1}^{k}m_{i}$ pendant vertices attached to the common vertex.
In this section, we will give our result about the minimum edge revised Szeged index in $\mathcal{C}(n, k)$.
Before that we need the following Lemma \ref{lem3.1} firstly.

\begin{lemma}\label{lem3.1}
For any graph $G \in \mathcal{C}(n, k)$, suppose that $C_{1}, C_{2}, \cdots, C_{k}$ are $k$ disjoint cycles of $G$ and $|E(G)|=m$. If $m_{i}=|E(C_{i})|$ for $i=1, 2, \cdots, k$
and $m_{1}, m_{2}, \cdots, m_{t}$ are odd integers, $m_{t+1}, m_{t+2}, \cdots, m_{k}$ are even integers. Then, we have that
$$Sz^{\ast}_{e}(G) \geq \frac{2m^{2}-m}{4}+\frac{(m-1)^{2}\sum\limits_{i=1}^{k}m_{i}}{4}-\frac{\sum\limits_{i=1}^{t}(m_{i}-1)(m-m_{i})^{2}}{4}-\frac{\sum\limits_{i=t+1}^{k}m_{i}(m-m_{i})^{2}}{4}$$
with equality if and only if $G \cong G(m_{1}, m_{2}, \cdots, m_{k})$.
\end{lemma}

\begin{proof} Let $E'$ be the set of all cut edges of $G$. Then $E'=E(G)\setminus \{\cup_{i=1}^{k}E(C_{i})\}$ and $|E'|=m-\sum\limits_{i=1}^{k}m_{i}$.
By Lemma \ref{lem2.1}, one has that
$$\sum\limits_{e=uv\in E'}(m_{u}(e|G)-m_{v}(e|G))^{2} \leq (m-\sum\limits_{i=1}^{k}m_{i})(m-1)^{2}$$
with equality if and only if all cut edges are pendant edges.

By Lemma \ref{lem2.2}, we have that for $i=t+1, t+2, \cdots, k$,
$$\sum\limits_{e=uv\in C_{i}}(m_{u}(e|G)-m_{v}(e|G))^{2} \leq m_{i}(m-m_{i})^{2}$$
with equality if and only if $C_{i}$ is an end block.

In view of Lemma \ref{lem2.3}, we have that for $j=1, 2, \cdots, t$,
$$\sum\limits_{e=uv\in C_{j}}(m_{u}(e|G)-m_{v}(e|G))^{2} \leq (m_{j}-1)(m-m_{j})^{2}$$
with equality if and only if $C_{j}$ is an end block.

Thus, we have
\begin{eqnarray*}
Sz^{\ast}_{e}&=&\frac{m^{3}}{4}-\frac{1}{4}\sum\limits_{e=uv\in E(G)}(m_{u}(e|G)-m_{v}(e|G))^{2}\\
&\geq&\frac{m^{3}}{4}-\frac{(m-\sum\limits_{i=1}^{k}m_{i})(m-1)^{2}}{4}-\frac{\sum\limits_{i=1}^{t}(m_{i}-1)(m-m_{i})^{2}}{4}-\frac{\sum\limits_{i=t+1}^{k} m_{i}(m-m_{i})^{2}}{4}\\
&=&\frac{2m^{2}-m}{4}+\frac{(m-1)^{2}\sum\limits_{i=1}^{k}m_{i}}{4}-\frac{\sum\limits_{i=1}^{t}(m_{i}-1)(m-m_{i})^{2}}{4}-\frac{\sum\limits_{i=t+1}^{k} m_{i}(m-m_{i})^{2}}{4}
\end{eqnarray*}
with equality if and only if all cut edges are pendant edges and all cycles are end-block i.e., $G \cong G(m_{1}, m_{2}, \cdots, m_{k})$.
\end{proof}

\begin{theorem}\label{T3.2}
Let $G \in \mathcal{C}(n, k)$ and $|E(G)|=m$, the following statements holds:

(i) If $m \geq 15$ and $m \geq 4k$, then $Sz^{\ast}_{e}(G) \geq \frac{2m^{2}-m}{4}+k(6m-15)$ with equality if and only if $G \cong \mathcal{C}_{1}(n, k)$;

(ii) If $m \geq 15$ and $m < 4k$, then $Sz^{\ast}_{e}(G) \geq \frac{2m^{2}-m+(4k-m)[(m-9)^{2}-36]}{4}+k(6m-15)$
with equality if and only if $G \cong G(  \underbrace{3, \cdots, 3}_{4k-m}, \underbrace{4, \cdots, 4}_{m-3k})$;

(iii) If $m < 15$, then $Sz^{\ast}_{e}(G) \geq \frac{2m^{2}-m+k[(m-9)^{2}-36]}{4}+k(6m-15)$ with equality if and only if $G \cong \mathcal{C}_{0}(n, k)$.

\end{theorem}

\begin{proof}
Suppose that $C_{1}, C_{2}, \cdots, C_{k}$ are $k$ disjoint cycles of $G$ and  $m_{i}=|E(C_{i})|$ for $i=1, 2, \cdots, k$. Without loss of generality, we may therefore assume that $m_{1}, m_{2}, \cdots, m_{t}$ are odd and $m_{t+1}, m_{t+2},\\ \cdots, m_{k}$ are even. According to Lemma \ref{lem3.1}, we have that
$$Sz^{\ast}_{e}(G) \geq Sz^{\ast}_{e}(G(m_{1}, m_{2}, \cdots, m_{k})).$$

Set $f(m_{1}, m_{2}, \cdots, m_{k})=Sz^{\ast}_{e}(G(m_{1}, m_{2}, \cdots, m_{k}))$. Then it is routine to check that
$$f(m_{1}, m_{2}, \cdots, m_{k})= \frac{2m^{2}-m}{4}+\frac{(m-1)^{2}\sum\limits_{i=1}^{k}m_{i}}{4}-\frac{\sum\limits_{i=1}^{t}(m_{i}-1)(m-m_{i})^{2}}{4}-\frac{\sum\limits_{i=t+1}^{k} m_{i}(m-m_{i})^{2}}{4}$$
and
$$ \frac{\partial f(m_{1}, m_{2}, \cdots, m_{k})}{\partial m_{i}}=\left\{
                                                 \begin{array}{ll}
                                                   \frac{(m-1)^{2}-(m-3m_{i}+2)(m-m_{i})}{4}>0, & \hbox{$1 \leq i \leq t$ ;} \\
                                                   \frac{(m-1)^{2}-(m-3m_{i})(m-m_{i})}{4}>0, & \hbox{$t+1 \leq i \leq k$ .}
                                                 \end{array}
                                               \right.
  $$
So
$$f(m_{1}, m_{2}, \cdots, m_{k}) \geq f(\underbrace{3, \cdots, 3}_{t}, \underbrace{4, \cdots, 4}_{k-t}).$$

Assume $g(t)=f(\underbrace{3, \cdots, 3}_{t}, \underbrace{4, \cdots, 4}_{k-t})$, it follows that
$$g(t)= \frac{2m^{2}-m}{4}+k(6m-15)+\frac{t[(m-9)^{2}-36]}{4}$$
and
$$g'(t)=\frac{(m-9)^{2}-36}{4}.$$

Note that if $m \geq 15$ and $m-4k<0$, then there is at least $r$ triangles, where $3r+4(k-r)=m$, i.e., $r=4k-m$. So we have that

$\bullet$ If $m \geq 15$ and $m \geq 4k$, then $g'(t)\geq 0$ and $Sz^{\ast}_{e}(G) \geq  g(0)= \frac{2m^{2}-m}{4}+k(6m-15)$ with equality if and only if $G \cong \mathcal{C}_{1}(n, k)$.

$\bullet$ If $m \geq 15$ and $m < 4k$, then $g'(t)\geq 0$ and $Sz^{\ast}_{e}(G) \geq g(4k-m)= \frac{2m^{2}-m+(4k-m)[(m-9)^{2}-36]}{4}+k(6m-15)$
with equality if and only if $G \cong G(  \underbrace{3, \cdots, 3}_{4k-m}, \underbrace{4, \cdots, 4}_{m-3k})$.

$\bullet$ If $m < 15$, then $g'(t) < 0$ and $Sz^{\ast}_{e}(G) \geq g(k)= \frac{2m^{2}-m+k[(m-9)^{2}-36]}{4}+k(6m-15)$ with equality if and only if $G \cong \mathcal{C}_{0}(n, k)$.

The result follows.
\end{proof}

\section{Cactus with second minimum edge revised Szeged index in $\mathcal{C}(n, k)$}

In this section, we will determine the graphs in $\mathcal{C}(n, k)$ with the second minimum edge revised Szeged index. In what follows,
assume that $m > 15$ and $m > 4k$, and $f(m_{1}, m_{2}, \cdots, m_{k})$ is same as the proof of Theorem \ref{T3.2}.  We will give the lower bound of $Sz^{\ast}_{e}(G)$ in $\mathcal{C}(n, k)\setminus  \mathcal{C}_{1}(n, k) $ and
determine the corresponding extremal graph. Let $G^{\ast}_{1}$ be the graph that is obtained from $\mathcal{C}_{1}(n-1, k)$ by adding a pendant edge at the pendant
vertex of $\mathcal{C}_{1}(n-1, k)$, see Fig. 2.

\vspace{0.8cm}

\begin{center}   \setlength{\unitlength}{0.7mm}
\begin{picture}(30,65)

\put(-15,20){\circle*{1}}
\put(-5,20){\circle*{1}}
\put(5,20){\circle*{1}}
\put(30,20){\circle*{1}}
\put(15,40){\circle*{1}}
\put(-5,60){\circle*{1}}
\put(-15,60){\circle*{1}}
\put(0,60){\circle*{1}}
\put(10,60){\circle*{1}}
\put(35,60){\circle*{1}}
\put(25,60){\circle*{1}}

\put(0,30){\circle*{1}}

\put(-10,65){\circle*{1}}
\put(5,65){\circle*{1}}
\put(30,65){\circle*{1}}
\put(-15,60){\line(1,1){5.1}}
\put(-5,60){\line(-1,1){5.1}}

\put(0,60){\line(1,1){5.1}}
\put(10,60){\line(-1,1){5.1}}

\put(25,60){\line(1,1){5.1}}
\put(35,60){\line(-1,1){5.1}}

\put(15,40){\line(-3,-2){30}}
\put(15,40){\line(-1,-1){20}}
\put(15,40){\line(-1,-2){10}}
\put(15,40){\line(3,-4){15}}
\put(15,40){\line(-3,2){30}}
\put(15,40){\line(-1,1){20}}
\put(15,40){\line(-3,4){15}}
\put(15,40){\line(-1,4){5}}
\put(15,40){\line(1,2){10}}
\put(15,40){\line(1,1){20}}
\put(11,18){$\cdots$}
\put(18,18){$ \cdots $}
\put(14,58.5){$\cdots$}
\put(-32,14){$n-3k-2$ pendant vertices}
\put(-3,70){$k$ quadrangles}

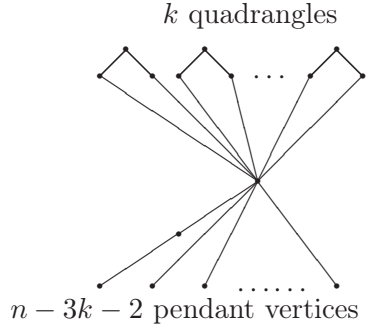
\put(-10,5){Fig. 2. The figure $G^{\ast}_{1}$}

\end{picture} \end{center}

\vspace{0.1cm}

\begin{lemma}\label{lem4.1}
Let $G$ be a graph in $\mathcal{C}(n, k)\setminus  \mathcal{C}_{1}(n, k) $ such that there exists a cut edge that is not a
pendant edge and $|E(G)|=m$. Then
$$Sz^{\ast}_{e}(G) \geq Sz^{\ast}_{e}(G(\underbrace{4, \cdots, 4}_{k}))+m-2=f(\underbrace{4, \cdots, 4}_{k})+m-2$$
with equality if and only if $G \cong G^{\ast}_{1}$.
\end{lemma}

\begin{proof}
Suppose that $C_{1}, C_{2}, \cdots, C_{k}$ are $k$ cycles of $G$ and  $m_{i}=|E(C_{i})|$ for $i=1, 2, \cdots, k$. Without loss of generality, we may assume that $m_{1}, m_{2}, \cdots, m_{t}$ are odd and $m_{t+1}, m_{t+2}, \cdots, m_{k}$ are even.

Let $E'$ be the set of all cut edges of $G$.
Suppose that $e_{1}=xy$ is a cut edge that is not a pendant edge. Let $G_{x}$ and $G_{y}$ be the components of $G-xy$ that contain $x$ and $y$, respectively.
Then, it can be checked that
$$(m_{x}(e_{1}|G)-m_{y}(e_{1}|G))^{2}=(|E(G_{x})|-|E(G_{y})|)^{2}\leq (m-3)^{2},$$
with equality if and only if $G_{x}$ or $G_{y}$ is an edge. By Lemma \ref{lem2.1}, we have that
$$\sum\limits_{e=xy \in E'}(m_{x}(e|G)-m_{y}(e|G))^{2} \leq (m-\sum\limits_{i=1}^{k}m_{i}-1)(m-1)^{2}+(m-3)^{2}.$$

So we have that
\begin{eqnarray*}
Sz^{\ast}_{e}(G) &=&\frac{m^{3}}{4}-\frac{\sum\limits_{e=uv \in G}(m_{u}(e|G)-m_{v}(e|G))^{2}}{4}\\
&\geq&\frac{m^{3}}{4}-\frac{(m-\sum\limits_{i=1}^{k}m_{i}-1)(m-1)^{2}}{4}-\frac{(m-3)^{2}}{4}\\
&-&\frac{\sum\limits_{i=1}^{t}(m_{i}-1)(m-m_{i})^{2}}{4}-\frac{\sum\limits_{i=t+1}^{k} m_{i}(m-m_{i})^{2}}{4}\\
&=&f(m_{1}, m_{2}, \cdots, m_{k})+\frac{(m-1)^{2}}{4}-\frac{(m-3)^{2}}{4}\\
&\geq&f(\underbrace{4, \cdots, 4}_{k})+m-2.
\end{eqnarray*}

The equality of first inequality holds if and only if all cycles are end-block and all cut edges but $xy$ are pendant edges, and $xy$ has the property that $G-xy$
has a component that is an edge. The equality in second inequality holds if and only if all cycles are $C_{4}$.

Thus, we have that if there exists a cut edge that is not a pendant edge, then
$$Sz^{\ast}_{e}(G) \geq f(\underbrace{4, \cdots, 4}_{k})+m-2$$
with equality if and only if $G \cong G^{\ast}_{1}$.
\end{proof}

\begin{lemma}\label{lem4.2}
Let $G$ be a graph in $\mathcal{C}(n, k)\setminus  \mathcal{C}_{1}(n, k) $ such that there exists a cycle that is not $C_{4}$ and $|E(G)|=m$.

(i) If $G$ has an odd cycle, then $Sz^{\ast}_{e}(G) \geq f(3, \underbrace{4, \cdots, 4}_{k-1}) $ with equality if and only if $G \cong G(3, \underbrace{4, \cdots, 4}_{k-1})$.

(ii) If all cycles of $G$ are even, then $Sz^{\ast}_{e}(G) \geq f(6, \underbrace{4, \cdots, 4}_{k-1}) $.

\end{lemma}

\begin{proof}
Suppose that $C_{1}, C_{2}, \cdots, C_{k}$ are $k$ cycles of $G$ and  $m_{i}=|E(C_{i})|$ for $i=1, 2, \cdots, k$. Without loss of generality, assume that $m_{1}, m_{2}, \cdots, m_{t}$ are odd and $m_{t+1}, m_{t+2}, \cdots, m_{k}$ are even.

If $t \geq 1$, with the method similar to the proofs of Lemma \ref{lem3.1} and Theorem \ref{T3.2}, one has that
$$Sz^{\ast}_{e}(G) \geq  f(m_{1}, m_{2}, \cdots, m_{k}) \geq   f(\underbrace{3, \cdots, 3}_{t}, \underbrace{4, \cdots, 4}_{k-t})  \geq f(3, \underbrace{4, \cdots, 4}_{k-1}).$$
The last inequality holds since $g(t)=f(\underbrace{3, \cdots, 3}_{t}, \underbrace{4, \cdots, 4}_{k-t}) $ is an increasing function on $t$ when $m > 15$.

Hence, we have that if $G$ has an odd cycle, then
$$Sz^{\ast}_{e}(G) \geq f(3, \underbrace{4, \cdots, 4}_{k-1}) $$
 with equality if and only if $G \cong G(3, \underbrace{4, \cdots, 4}_{k-1})$.

If $t=0$, then $m_{i}$ are even for all $i=1, 2, \cdots, k$. As there exists a cycle that is not $C_{4}$, without loss of generality, assume that
$m_{1} \geq6$. Since $f(m_{1}, m_{2}, \cdots, m_{k})$ is an increasing function on $m_i$, one has that
$$Sz^{\ast}_{e}(G) \geq   f(m_{1}, m_{2}, \cdots, m_{k})  \geq f(6, \underbrace{4, \cdots, 4}_{k-1}) .$$

The proof is finished
\end{proof}

\begin{lemma}\label{lem4.3}
Let $G$ be a graph in $\mathcal{C}(n, k)\setminus  \mathcal{C}_{1}(n, k) $ such that there exists a cycle that is not an end-block and $|E(G)|=m$. Then
$$Sz^{\ast}_{e}(G) \geq {\rm{min}} \{ f(3, \underbrace{4, \cdots, 4}_{k-1}), f(\underbrace{4, \cdots, 4}_{k})+2m-10 \}. $$

\end{lemma}

\begin{proof} We can first assume that all the cycles of $G$ are $C_{4}$. Otherwise, by Lemma \ref{lem4.2},
we have that $Sz^{\ast}_{e}(G) \geq f(3, \underbrace{4, \cdots, 4}_{k-1}) $
or $Sz^{\ast}_{e}(G) \geq f(6, \underbrace{4, \cdots, 4}_{k-1}) > f( \underbrace{4, \cdots, 4}_{k})+2m-10$, our conclusion follows immediately.

Now let $C_{1}=v_{1}v_{2}v_{3}v_{4}v_{1}$ be the cycle that is not an end-block, i.e., at least two of $v_{1}, v_{2}, v_{3}, v_{4}$ have degree more than 2.
Without loss of generality, let $d_{G}(v_{1}) \geq 3$. Assume that $ \{ v_{2}, v_{4}, x \} \subseteq N_{G}(v_{1})$.

If $d_{G}(v_{2}) \geq 3$ or $d_{G}(v_{4}) \geq 3$, without loss of generality, assume $d_{G}(v_{2}) \geq 3$ and $ \{ v_{1}, v_{3}, y \} \subseteq N_{G}(v_{2})$. This implies that $ \{ v_{1}v_{4}, v_{1}x \} \subseteq M_{v_{1}}(v_{1}v_{2}|G)$
and $ \{ v_{2}v_{3}, v_{2}y \} \subseteq M_{v_{2}}(v_{1}v_{2}|G)$. Similarly, one has that $ \{ v_{1}v_{4}, v_{1}x \} \subseteq M_{v_{4}}(v_{3}v_{4}|G)$ and
$ \{ v_{2}v_{3}, v_{2}y \} \subseteq M_{v_{3}}(v_{3}v_{4}|G)$. We obtain that
$$\sum\limits_{e=uv \in C_{1}} (m_{u}(e|G)-m_{v}(e|G))^{2}\leq 2(m-4)^{2}+2(m-6)^{2}=4(m-4)^{2}-8(m-5).$$

If $d_{G}(v_{2}) = d_{G}(v_{4}) =2$. It implies that $d_{G}(v_{3}) \geq 3$ and one can assume $ \{ v_{2}, v_{4}, z \} \subseteq N_{G}(v_{3})$. It is obvious that $ \{ v_{1}v_{4}, v_{1}x \} \subseteq M_{v_{1}}(v_{1}v_{2}|G)$
and $ \{ v_{2}v_{3}, v_{3}z \} \subseteq M_{v_{2}}(v_{1}v_{2}|G)$. Similarly, we have that $ \{ v_{1}v_{2}, v_{1}x \} \subseteq M_{v_{2}}(v_{2}v_{3}|G)$ and
$ \{ v_{3}v_{4}, v_{3}z \} \subseteq M_{v_{3}}(v_{2}v_{3})$; $ \{ v_{2}v_{3}, v_{3}z \} \subseteq M_{v_{3}}(v_{3}v_{4}|G)$ and
$ \{ v_{1}v_{4}, v_{1}x \} \subseteq M_{v_{4}}(v_{3}v_{4}|G)$; $ \{ v_{1}v_{2}, v_{1}x \} \subseteq M_{v_{1}}(v_{1}v_{4}|G)$ and
$ \{ v_{3}v_{4}, v_{3}z \} \subseteq M_{v_{4}}(v_{1}v_{4}|G)$. Then
$$\sum\limits_{e=uv \in C_{1}} (m_{u}(e|G)-m_{v}(e|G))^{2}\leq 4(m-6)^{2} \leq 4(m-4)^{2}-8(m-5).$$

By Lemmas \ref{lem2.1} and \ref{lem2.2}, one has that
\begin{eqnarray*}
Sz^{\ast}_{e}(G) &=&\frac{m^{3}}{4}-\frac{\sum\limits_{e=uv \in G}(m_{u}(e|G)-m_{v}(e|G))^{2}}{4}\\
&\geq&\frac{m^{3}}{4}-\frac{(m-4k)(m-1)^{2}}{4}- \frac{\sum\limits_{i=2}^{k} 4(m-4)^{2} }{4}-\frac{4(m-4)^{2}-8(m-5)}{4}\\
&=& f(\underbrace{4, \cdots, 4}_{k})+2m-10.
\end{eqnarray*}

So we have our conclusion.
\end{proof}

By Lemmas \ref{lem4.1}-\ref{lem4.3}, Theorem \ref{T4.4} holds.

\begin{theorem}\label{T4.4}
Let $G$ be a graph in $\mathcal{C}(n, k)\setminus  \mathcal{C}_{1}(n, k) $ and $|E(G)|=m$ with $m > 15$ and $m >4k$. Then
$$Sz^{\ast}_{e}(G) \geq \left\{
                                                 \begin{array}{ll}
                                                   \frac{2m^{2}+3m+4k(6m-15)-8}{4}, & \hbox{$m \geq 20$ ;} \\
                                                   \frac{3m^{2}-19m+4k(6m-15)+45}{4}, & \hbox{$15 < m \leq19$. }
                                                 \end{array}
                                               \right.
  $$
With equality if and only if $G\cong G^{\ast}_{1}$ for $m \geq 20$, and $G\cong G(3, \underbrace{4, \cdots, 4}_{k-1})$ for $15 < m \leq19$, where $G^{\ast}_{1}$
is show in Fig 2.
\end{theorem}

\begin{proof}

For any $G \in \mathcal{C}(n, k)\setminus  \mathcal{C}_{1}(n, k) $, one of the following three conditions hold:

$\bullet$ $G$ has a cut edge that is not a pendant edge;

$\bullet$ There is a cycle that is not a $C_{4}$;

$\bullet$ There is a cycle that is not an end-block.

By Lemmas \ref{lem4.1}-\ref{lem4.3}, we can see that
$$Sz^{\ast}_{e}(G) \geq {\rm{min}} \{ f(\underbrace{4, \cdots, 4}_{k})+m-2, f(3, \underbrace{4, \cdots, 4}_{k-1}) \}. $$

Since
$$f(3, \underbrace{4, \cdots, 4}_{k-1})-(f(\underbrace{4, \cdots, 4}_{k})+m-2)=\frac{m^{2}-22m+53}{4}.$$
The following result hold.

$\bullet$ If $m \geq 20$, then $f(3, \underbrace{4, \cdots, 4}_{k-1})-(f(\underbrace{4, \cdots, 4}_{k})+m-2)> 0$ and
$Sz^{\ast}_{e}(G) \geq f(\underbrace{4, \cdots, 4}_{k})+m-2=\frac{2m^{2}+3m+4k(6m-15)-8}{4}$ with equality if and only if $G\cong G^{\ast}_{1}$.

$\bullet$ If $15 < m \leq 19$, then $Sz^{\ast}_{e}(G) \geq f(3, \underbrace{4, \cdots, 4}_{k-1})=\frac{3m^{2}-19m+4k(6m-15)+45}{4}$
with equality if and only if $G\cong G(3, \underbrace{4, \cdots, 4}_{k-1})$.

The result follows.
\end{proof}

\section{Conclusions}

In this paper, the lower bound on edge revised Szeged index of the cacti with $n$ vertices and $k$ cycles is determined and the corresponding extremal graph is identified. Furthermore, the second minimal edge revised Szeged index of the cacti with given cycles is established as well.
For further study, it would be interesting to determine the extremal graph that has the maximum edge revised Szeged index in these class of cacti. Moreover, it would be meaningful to study the edge revised szeged index of other kinds of graphs.

\section*{Acknowledgments}

This work was supported by the National Natural Science Foundation of China
(No. 11731002), the Fundamental Research Funds for the Central Universities (Nos. 2016JBM071, 2016JBZ012) and the $111$ Project of China (B16002).


\end{document}